\documentclass[12pt, a4paper]{article}
\makeatletter
\newenvironment{folding}{\endgroup}{\begingroup \def \@currenvir{folding}\edef \@currenvline{\on@line}}
\makeatother
\begin{folding} 
	\usepackage[utf8]{inputenc}
	\usepackage{hyperref }
	\usepackage{latexsym,amssymb, graphicx, bbm, amsmath}
	\usepackage{amsthm, amsfonts, verbatim}
	\usepackage{xcolor}
	\usepackage{pifont}
	\newcommand{\cmark}{\ding{51}}%
	\newcommand{\xmark}{\ding{55}}%
\end{folding}

\begin{folding} 
	\usepackage{tikz}
	\usetikzlibrary{shapes.geometric, arrows}
	\tikzstyle{rectwhite} = [rectangle, rounded corners, minimum width=1cm, minimum height=1cm,text centered, draw=black, fill = pink!30]
	\tikzstyle{arrow} = [thick,->,>=stealth]
\end{folding}

\begin{folding} 
	\usepackage{amsthm, amsfonts, amsmath, amssymb}
	\usepackage{wasysym}
	\usepackage{indentfirst}
	\usepackage{color}
	\usepackage{graphicx}
	\usepackage{microtype}

	\DeclareMathOperator{\hol}{\text{Hol}}

	\newcommand{\T}{\mathbb{T}}
	\newcommand{\D}{\mathbb{D}}

	\newcommand{\dd}{\, \mathrm{d}}
	\newcommand{\ddd}{\mathrm{d}}
	\newcommand{\abs}[1]{\vert #1 \vert}				
	\newcommand{\labs}[1]{\left\vert #1 \right\vert}	
	\newcommand{\norm}[1]{\Vert #1 \Vert}				
	\newcommand{\lnorm}[1]{\left\Vert #1 \right\Vert}	
\end{folding}
\graphicspath{{../figs/}}

\begin{folding} 
	\newtheoremstyle{dotless}{}{}{}{}{\bfseries}{}{12pt}{}
	\theoremstyle{dotless}
	\newtheorem{thm}{Theorem}[section] 
	\newtheorem*{thm*}{Theorem} 
	\newtheorem{lem}[thm]{Lemma} 
	\newtheorem*{lem*}{Lemma} 
	\newtheorem{cor}[thm]{Corollary} 
	\newtheorem*{cor*}{Corollary} 
	\newtheorem{prop}[thm]{Proposition} 
	\newtheorem*{prop*}{Proposition} 
	\newtheorem{defn}[thm]{Definition} 
	\newtheorem*{defn*}{Definition} 
	\newtheorem{rem}[thm]{Remark} 
	\newtheorem*{rem*}{Remark} 
	\newtheorem*{exa*}{Example} 
	\newtheorem*{exe*}{Exercise} 
	\newtheorem*{aside*}{Aside} 
	\newtheorem{example}[thm]{Example} 
	\newtheorem*{example*}{Example} 
	\newtheorem*{notation*}{Notation} 
\end{folding}
\begin{folding} 
	
\end{folding}

\newcommand{\dee}{\mathrm{d}}

\title{
	{Zero-free half-planes of the $\zeta$-function via spaces of analytic functions}\\ \vspace{10mm}
}

\author{Aditya Ghosh \\Kobi Kremnizer\\S. Waleed Noor\\ Charles F. Santos  }

\begin{document}
	\pagenumbering{gobble}
	\maketitle
	\begin{abstract}
		In this article we introduce a general approach for deriving zero-free half-planes for the Riemann zeta function $\zeta$ by identifying topological vector spaces of analytic functions with specific properties. This approach is applied to weighted $\ell^2$ spaces and the classical Hardy spaces $ H^p $ ($ 0<p\leq2 $). As a consequence precise conditions are  obtained for the existence of zero-free half planes for the $\zeta$-function. 
	\end{abstract}
	\pagebreak
	\pagenumbering{arabic}
\section{Introduction} \label{sec: introduction}

The Riemann Hypothesis (RH) is equivalent to a completeness problem in $ L^2(0,1) $. This was first stated in the 1950's by Nyman \cite{nyman:paper} and Beurling \cite{beurling:closureproblempaper}:

\begin{thm} \label{Beurling L2(0,1) equivalence}
Define $\rho_\alpha(x) := \rho\left(\frac{\alpha}{x}\right) - \alpha\rho\left(\frac{1}{x}\right) $ for $0<\alpha<1$ (where $\rho(x)$ denotes the fractional part of $x$). If we denote $\nu := \mathrm{span}\{\rho_\alpha \mid 0<\alpha<1\}$ and $\mathbbm{1}_{(0,1)}$ the characteristic function of $(0,1)$, then the following statements are equivalent:
\begin{enumerate}
\item The RH holds true,

		\item $\mathbbm{1}_{(0,1)}$ belongs to the closure of $\nu$ in $L^2(0,1)$,

		\item $\nu$ is dense in $ L^2(0,1) $.
\end{enumerate}
\end{thm}

A remarkable strengthening by Báez-Duarte \cite{bd:strenghteningpaper} in 2003 showed it is enough to restrict $\nu$ to the countable index set $\alpha=1/\ell$ for $\ell\in\mathbb{N}$ in condition 2 (whereby the density of $\nu$ in condition $3$  no longer holds). The reader is directed to an article by Bagchi \cite{bagchi:paper} that collects these results of Nyman, Beurling and Báez-Duarte and their proofs in one place. Recently in \cite{noor:paper} these ideas have been transferred to the Hardy-Hilbert Space $ H^2(\mathbb{D}) $. Here, the RH is equivalent to  the constant function $ \mathbf{1} $ being in the closed linear span of certain elements $ \{h_k \mid k \geq 2\} $ in $ H^2(\mathbb{D}) $. In \cite{noor:paper} it is also proved that  
\begin{equation*} 
\sum\limits_{k=2}^{n} \mu(k)(I-S)h_k \to 1-z \ \ \text{ as } n \to \infty
\end{equation*}
in $ H^2(\mathbb{D}) $, where $ S$ is the shift operator and $\mu$ the Möbius function. As a  consequence, this proves the density of ${\mathrm{span}}\{h_k\mid k \geq2\} $ in $H^2(\mathbb{D})$ in the compact-open topology (weaker than the $H^2$-topology) by relating it with invertibility of $ (I-S) $. The goal of this paper is to generalize these ideas to other spaces of analytic functions and establish criteria that would guarantee zero-free regions for the $\zeta$-function. 

The plan of the paper is the following. After a section of preliminaries, our general framework for obtaining zero-free half-planes for $\zeta$ is introduced in Section 2. This approach entails finding topological vector spaces of analytic functions $X$ that satisfy a \emph{checklist} of conditions. This general framework is then applied to the weighted Hardy spaces (unitarily equivalent to weighted $\ell^2$ spaces) in Section 3. After considering some concrete examples in Subsection \ref*{Some examples}, we obtain numerical conditions on the weights that guarantee zero-free regions for $\zeta$ in Subsection 3.2. In Subsection 3.3, we see that for non-trivial zero free half-planes the required weights can exhibit very extreme behavior. In Section 4, we apply this approach to the classical Hardy spaces $H^p$ for $0<p\leq 2$. In Subsection 4.1, we show that the checklist is completely satisfied for $H^p$ with $0<p<1$. In particular we prove that $\mathrm{span}\{h_k\mid k \in k\geq 2\} $ is dense in $H^p$ (see Corollary \ref{prop density}) for $ 0<p <1 $.  Since convergence in $H^p$ implies convergence in the compact-open topology, this strengthens one of the main results of \cite[Theorem 10]{noor:paper}.  In Subsection 4.2 all but one condition is satisfied for $H^p$ with $1\leq p\leq2$. More precisely, we show that 
\[
1 \in \overline{\mathrm{span}}_{H^p}\{h_k\mid k\geq 2\} \implies \zeta(s) \neq 0  \ \  \mathrm{for} \ \   \Re(s)> 1/p 
\]
(see Theorem \ref{Half-plane 1<p<2}) where the $h_k$ are defined in Section 1. This is an $H^p$ analogue of Beurling's results from\cite{beurling:closureproblempaper}. 

Before we begin, it is worthwhile to highlight some results in the literature, with the hope of making evident the continuity of ideas and techniques after relocating the RH and questions surrounding zero-free regions of the $\zeta$-function from the $L^p$ spaces to analytic function spaces. 
Beurling in \cite{beurling:closureproblempaper} proved that for every $1<p<\infty$, the non-vanishing of $\zeta(s)$ for $\Re(s)>1/p$ is equivalent to the density of $\nu$ in $L^p(0,1)$, hence generalizing Theorem \ref{Beurling L2(0,1) equivalence} considerably. Balazard and Saias \cite{Balazard-Saias1} continued the study of the relation between zero-free half-planes and approximation problems in $L^p$ spaces. Bercovici and Foias \cite{Bercovici-Foias} proved that the $L^2(0,1)$-closure of $\nu$ equals
\[
\{f\in L^2(0,1):\frac{\mathcal{F}f(s)}{\zeta(s)} \ \mathrm{is \ holomorphic \ for} \ \Re(s)>1/2\}
\]
where $\mathcal{F}$ denotes the Mellin transform which is an isometric isomorphism of $L^2(0,1)$ onto the Hardy space $H^2(\mathbb{C}_{1/2})$ of the half-plane $\Re(s)>1/2$. This formula may be viewed as an unconditional version of Theorem \ref{Beurling L2(0,1) equivalence}. See \cite{Balazard-SaiasExpo} for interesting discussions around this formula and its possible generalizations to other $L^p(0,1)$ spaces. The transform $\mathcal{F}$ has played an important role in this theory due to the identity
\[
\mathcal{F}\rho_\alpha(s)=\frac{\zeta(s)}{s}(\alpha-\alpha^s) \ \ \ \ \ \  \ (0<\alpha<1, \Re(s)>0).
\]
The starting point of our approach is to introduce functionals $\Lambda^{(s)}$ on spaces of analytic functions on $\mathbb{D}$ which replicate the role of $\mathcal{F}$ (see Section 2). The two families of spaces we apply $\Lambda^{(s)}$ to are the weighted sequence spaces $\ell^2_w$ (Section 3) and the Hardy spaces $H^p$ for $p>0$ (Section 4). In particular we obtain $H^p$ analogues of Beurling's results from\cite{beurling:closureproblempaper}. The Hardy space $H^{1/3}$ was employed by Balazard \cite{Balazard} to prove the formula
\[
\frac{1}{2\pi}\int_{\Re(s)=\frac{1}{2}}\frac{\log|\zeta(s)|}{|s|^2}|ds|=\sum_{\Re(\rho)>1/2}\log\left|\frac{\rho}{1-\rho}\right|
\]
where the sum is taken over the zeros of $\zeta$ (counting multiplicities) to the right of the critical line. In conclusion, spaces of analytic functions and Hardy spaces in particular have played an important role within the Nyman-Beurling approach to the RH and the theory it has inspired. Our goal is to explore further these connections. Balazard's bibliographical survey \cite{BalazardBibliography} contains detailed discussions on numerous works throughout the 20th century regarding completeness problems and a functional approach to the RH.

\section{Preliminaries} \label{sec: background}
	
	\begin{defn} \label{def H^2D}
		The Hardy-Hilbert space $ H^2(\mathbb{D}) $ consists of all holomorphic function on the unit disk $ \mathbb{D}$ that satisfy
		\begin{equation}\label{eqn def H^2D integral version}
		\norm{f} :=\left( \sup_{0<r<1} \int_\T \abs{f(rz)}^2 \dd m(z)\right)^{1/2}<\infty
		\end{equation}
		where $m$ is the normalized Lebesgue measure on $\mathbb{T}$. The space $H^2(\mathbb{D}) $ is a Hilbert space which inherits its inner product from the sequence space $\ell^2$:
		\begin{equation}\label{eqn def H^2D sequence space}
			H^2(\mathbb{D}) = \{f =\sum_{n=0}^{\infty} a_n z^n \mid (a_n)_{n\in\mathbb{N}} \in \ell^2\}.
		\end{equation}
\end{defn} 
The functions $h_k$ in $H^2(\mathbb{D}) $ were defined in \cite{noor:paper} as
	\[
	h_k(z) := \frac{1}{k} \frac{1}{1-z} \log \left( \frac{1+z+\ldots+z^{k-1}}{k} \right) \ , \ \ k\geq 2.
	\]
	Actually this definition of $h_k$ differs from that of $\cite{noor:paper}$ by the factor $1/k$. The $H^2({\mathbb{D}})$ version of Báez-Duarte's result in \cite{noor:paper} plays a central role in this work.
	\begin{thm}\label{BD H^2}
		The following statements are equivalent:
		\begin{enumerate}
			\item RH holds true,
			\item $\mathbf{1} $ belongs to the closure of ${\mathrm{span}}\{h_k\mid k \geq2\} $ in $H^2(\mathbb{D})$, and 
			\item ${\mathrm{span}}\{h_k\mid k \geq2\} $ is dense in $H^2(\mathbb{D})$.
		\end{enumerate}
		where $\mathbf{1}  $ is the constant function. 
	\end{thm} 
		Let $S=M_z$ denote the shift operator on $ H^2(\mathbb{D}) $ of multiplication by $z$. Then in \cite[Lemma 11]{noor:paper} it is also proved that
		\begin{equation}\label{eq Noor convergence of (I-S)hk}
			\Big\| \sum\limits_{k=2}^{n} \mu(k)(I-S)h_k - (1-z) \Big\|_{H^2(\mathbb{D})} \longrightarrow  0 \qquad \text{ as } n \to \infty
		\end{equation}
	and as a consequence it was established that
		\begin{equation}\label{Closure problem of (I-S)hk}
			\overline{\mathrm{span}}_{H^2(\mathbb{D})}\{(I-S)h_k \mid k \geq 2\} = H^2(\mathbb{D}).
		\end{equation}
It is important to note that $I-S$ is not an invertible operator on $ H^2(\mathbb{D})$ (but has dense  range). If it were invertible, then $ (I-S)^{-1} $ applied to \eqref{eq Noor convergence of (I-S)hk} or \eqref{Closure problem of (I-S)hk} would prove the RH by Theorem \ref{BD H^2}. It is worth remarking that a result similar to \eqref{eq Noor convergence of (I-S)hk} appears in \cite[Section 13]{Balazard-Saias4}, but in the latter article the shift operator $S$ is defined on the weighted $\ell^2$ space with weights $(1/(\zeta(2)k^2))_k$ which is equivalent to multiplication by $z$ on a weighted Bergman space $\mathcal{A}$ (defined below). The shifts on $ H^2(\mathbb{D})$ and $\mathcal{A}$ are not equivalent, and hence \eqref{eq Noor convergence of (I-S)hk} is not an immediate consequence of results in \cite{Balazard-Saias4}. In fact, unlike the shift on $ H^2(\mathbb{D})$, even basic questions such as a characterization of the closed invariant subspaces of the Bergman shifts remain open problems (see \cite{TheoryBergmanspaces}).
		
		It is useful to locate the $H^2(\mathbb{D})$ version of Báez-Duarte's theorem among the cornucopia of spaces where it appears in the literature (see Figure 1). The space $\mathcal{A}$ is the Hilbert space of analytic functions $f(z)=\sum_{n=0}^\infty a_nz^n$ and $g(z)=\sum_{n=0}^\infty b_nz^n$ defined on $\mathbb{D}$ for which the inner product is given by
		\[
		\langle f,g\rangle:=\sum_{n=0}^\infty \frac{a_n\overline{b_n}}{(n+1)(n+2)}.
		\]
Then the maps $T:H^2(\mathbb{D})\to\mathcal{A}$ (defined in Figure 1) and $\Psi:\ell^2_{\omega}\to \mathcal{A}$ 
\[
\Psi:(x(1),x(2),\ldots)\longmapsto \sum_{n=0}^\infty x(n+1)z^n
	\]
are isometric isomorphisms, where $\ell^2_\omega$ is the weighted $\ell^2$-space with weights $\omega_n=\frac{1}{(n+1)(n+2)}$ corresponding to the coefficients of functions in $\mathcal{A}$. See \cite{noor:paper} for more details on $T$ and $\Psi$. Let $\mathcal{M}$ be the closed subspace of  $L^2(0,1)$ consisting of functions almost everywhere constant on the intervals $[\frac{1}{n+1},\frac{1}{n})$ for $n\geq 1$, and $H^2(\mathbb{C}_\frac{1}{2})$ the Hardy space of analytic functions $F$ on the half-plane $\mathbb{C}_\frac{1}{2}:=\{s\in\mathbb{C}: \Re(s)>1/2\}$ such that
\[
||F||^2:=\sup_{\sigma>\frac{1}{2}}\frac{1}{2\pi}\int_{-\infty}^\infty|F(\sigma+it)|^2\dd t <\infty.
\]
The maps $U:\mathcal{M}\to\mathcal{H}$ and the Mellin transform $\mathcal{F}:L^2(0,1)\to H^2(\mathbb{C}_\frac{1}{2})$ are also isometric isomorphisms. See \cite{bagchi:paper} for more details on $U$ and $\mathcal{F}$.
	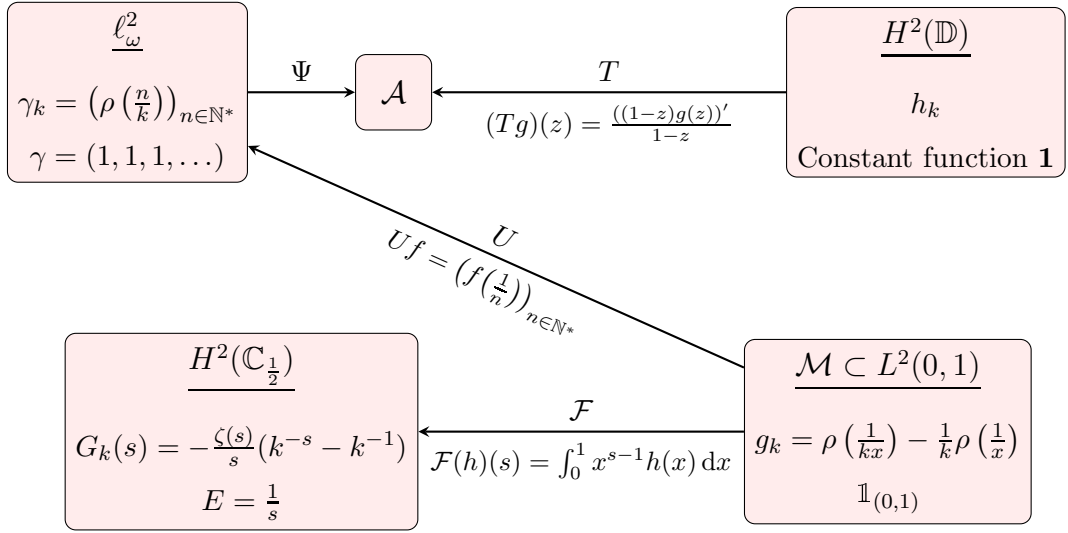
\begin{figure}[!h] 
	\begin{tikzpicture}[node distance=8.5cm]
	\node (H2O) [rectwhite, align=center] {\underline{$ H^2(\mathbb{C}_\frac{1}{2}) $} \\[1em] \small $ G_k(s) = - \frac{\zeta(s)}{s} (k^{-s} - k^{-1})$\\ [0.5em]
		\small $E = \frac{1}{s} $};
	
	\node (M) [rectwhite, align=center, right of=H2O, xshift = 0 cm]{\underline{$ \mathcal{M} \subset L^2(0,1) $} \\[1em] \small $ g_k = \rho\left(\frac{1}{kx}\right) - \frac{1}{k}\rho\left(\frac{1}{x}\right) $ \\[0.5em] \small $ \mathbbm{1}_{(0,1)} $ };
	
	\node (H) [rectwhite, align=center, above of = H2O, yshift = -4 cm, xshift = -1.5 cm]{\underline{$\ell^2_\omega$} \\ [1em] \small $ \gamma_k  = \left(\rho\left(\frac{n}{k}\right)\right)_{n\in\mathbb{N}^*} $ \\[0.5em] \small $ \gamma = (1,1,1,\ldots) $};
	
	\node (A) [rectwhite, align=center, right of=H, xshift = -5 cm]{$\mathcal{A}$};
	
	\node (D) [rectwhite, align=center, right of=A, xshift = -1.5 cm]{\underline{$ H^2(\mathbb{D}) $} \\[1em] \small $ h_k $ \\[0.5em] \small Constant function  $ \mathbf{1} $};
	
	\draw [arrow] (M) -- node[anchor=north]{\footnotesize $\mathcal{F}(h)(s)=\int_0^1x^{s-1}h(x)\dd x$} node[anchor=south]{\small$\mathcal{F}$} (H2O);
	
	\draw [arrow] (M)-- node[sloped, anchor=center, below]{\footnotesize $Uf = \left(f{\left( \frac{1}{n}\right)} \right)_{n\in\mathbb{N}^*}$}node[sloped, anchor=center, above]{\small $U$}(H);
	
	\draw [arrow] (D) -- node[anchor=north]{\footnotesize $ (Tg)(z) = \frac{\left((1-z)g(z)\right)^\prime}{1-z} $} node[anchor=south]{\small$T$} (A);
	\draw [arrow] (H) -- node[anchor=south]{\small$\Psi$} (A);
 ;
	\end{tikzpicture} 
	\caption{Spaces and isometries between them.}
	\label{fig:diagram of maps}
\end{figure} 
 
\section{General framework} 
\label{sec: general spaces}
In this section we outline our approach to finding zero free half-planes for $\zeta$ using more general spaces of analytic functions, but that contain $H^2(\mathbb{D})$. We begin by looking more carefully at the situation in $H^2(\mathbb{D})$.  

From Figure 1, we see that there is a chain of isometries:
\begin{equation}\label{eqn Chain of Isometries}
	H^2(\mathbb{D}) \xrightarrow[\cong]{T} \mathcal{A} \xrightarrow[\cong]{\Psi^{-1}}\ell^2_\omega \xrightarrow[\cong]{U^{-1}} \mathcal{M} \xrightarrow[\text{not onto}]{\mathcal{F}}H^2(\mathbb{C}_\frac{1}{2}). 
\end{equation}
We denote by $\Lambda \colon H^2(\mathbb{D}) \to H^2(\mathbb{C}_\frac{1}{2})$ the composition of these isometries.
Since $ H^2(\mathbb{C}_\frac{1}{2}) $ is a reproducing kernel Hilbert space, the evaluation functionals $E_{s} \colon H^2(\mathbb{C}_\frac{1}{2})\to \mathbb{C}$ for each $s \in \mathbb{C}_\frac{1}{2}$ are bounded. So if we define the functionals $\Lambda^{(s)}:=E_s\circ \Lambda: H^2(\mathbb{D})\to \mathbb{C}$, then we get
\begin{lem} \label{lem half plane on H^2(D)}
	$\Lambda^{(s)}$ is bounded on $H^2(\mathbb{D})$ for $\Re(s)>1/2 $.
\end{lem}
The principal feature of these  functionals is the property that
\begin{equation}\label{Principal feature of functionals}
\Lambda^{(s)}(h_k)=G_k(s) = - \frac{\zeta(s)}{s} (k^{-s} - k^{-1})
\end{equation}
which can be seen from Figure 1.
By \eqref{eqn Chain of Isometries} one can also check that
\begin{equation}\label{def functionals on H2D }
\Lambda^{(s)}(1)=-\frac{1}{s}, \\ \ \ \Lambda^{(s)}(z^k) = f_k(s): = - \frac{1}{s} \left( (k+1)^{1-s} - k^{1-s}\right).
\end{equation}
Indeed, since $T(z^k)=kz^{k-1}-\frac{z^k}{1-z}$ we have the sequence
\[
s_k:=\Psi^{-1}T(z^k)=(0,\ldots,0,k,-1,-1,\ldots)\in\mathcal{H}.
\]
where the $k$-th term of $s_k$ is $k$. Now $p_k:=U^{-1}s_k$ belongs to $\mathcal{M}$ such that $p_k(x)=s_k(n)$ for all $x\in [\frac{1}{n+1},\frac{1}{n})$ and its Mellin transform is 
\[
\mathcal{F}(p_k)(s)=\int_\frac{1}{k+1}^\frac{1}{k}k x^{s-1}\dd x-\int^\frac{1}{k+1}_0 x^{s-1}\dd x=f_k(s).
\]
Therefore $\Lambda^{(s)}(z^k) = f_k(s)$ and $\Lambda^{(s)}$ is uniquely determined by these values.

The following growth estimate for $f_k(s)$ will be  used frequently.

\begin{prop} \label{F_k estimate}
	For each $\Re(s)>0$, we have $|f_k(s)|\asymp k^{-\Re(s)}$ for all $k\geq 1$.
\end{prop}
\begin{proof} By the Fundamental Theorem of Calculus we have
\begin{equation}\label{eqn 4b }
|(k+1)^{1-s} - k^{1-s}| =|1-s| \left\lvert \int_k^{k+1} y^{-s}\dd y \right\rvert \asymp|1-s|k^{-\Re(s)}
\end{equation}
and the growth estimate easily follows.
\end{proof}
The following \emph{checklist} of conditions summarizes our abstract approach. Denote by $\mathbb{C}_r$ the half-plane $\{s \in \mathbb{C} \mid \Re(s) > r\}$.

\phantomsection \label{Checklist}
Suppose a topological vector space $ X $ of analytic functions satisfies the following conditions: 

\noindent \framebox{ 
	\begin{minipage}{14 cm} 
		
		\begin{enumerate}
			\item[\textbf{(C1)}] $ z^k \in X $ for $ k \in \mathbb{N} $ form a Schauder basis of $ X $.
			\item[\textbf{(C2)}] $  H^2(\D) \subseteq X$ with the relative topology weaker than that of $H^2(\D)$.
			\item[\textbf{(C3)}] $ 1 $ belongs to the closure of ${\mathrm{span}}\{h_k\mid k \geq2\} $ in $X$.
			\item[\textbf{(C4)}]  There exists $ r \in \mathbb{R}$ such that the functionals $\Lambda^{(s)}:X\to \mathbb{C}$ defined by
\begin{align*} 
				z^k &\mapsto f_k(s)=- \frac{1}{s} \left( (k+1)^{1-s} - k^{1-s}\right) && (k \in \mathbb{N^*}) \\
				1 &\mapsto -\frac{1}{s}
\end{align*}
are bounded on $X$ for all $s\in\mathbb{C}_r$. 
		\end{enumerate}
	\end{minipage}
}\\  

 The following result provides the justification for our general approach.
\begin{prop}\label{FP} If there exists a space of  analytic functions $ X $ satisfying the conditions above for some $r\in\mathbb{R}$, then $\zeta(s)\neq 0$ for all $s\in\mathbb{C}_r$.
	\end{prop}
 \begin{proof}
     By \textbf{(C1)} and \textbf{(C2)} it is clear that $\Lambda^{(s)}$ is determined by it values on $H^2(\mathbb{D})$. Therefore  by \textbf{(C3)}, \textbf{(C4)} and \eqref{Principal feature of functionals} it follows that $\Lambda^{(s)}(1)=-1/s$ can be approximated pointwise by linear combinations of 
\[
\Lambda^{(s)}(h_k)=- \frac{\zeta(s)}{s} (k^{-s} - k^{-1})
\]
for all $s\in\mathbb{C}_r$. Since $1/s$ has no zeros for $s\in\mathbb{C}_r$, the same must be true for $\Lambda^{(s)}(h_k)$ and hence for $\zeta(s)$.
 \end{proof}
  In the remainder of this article we apply this approach to the weighted sequence spaces $\ell^2_w$ and the Hardy spaces $H^p$. In particular, we investigate the extent to which these spaces satisfy conditions \textbf{(C1)} to \textbf{(C4)} . It will become evident in the following sections that condition \textbf{(C3)} poses the main challenge here. For instance Example \ref{example logn} and Theorem \ref{Half-plane 1<p<2} show that \textbf{(C3)} would imply $\zeta(s)\neq 0$ for $\mathbb{C}_r$ with $1/2 < r < 1$. An alternative route to proving \textbf{(C3)} is to show that the operator $ I-S$ is invertible on $ X $. This is because the approximation \eqref{eq Noor convergence of (I-S)hk} holds in X by \textbf{(C2)}, that is 
	 \begin{equation}\label{eq Noor eqn in general X}
	 	\Big\| \sum\limits_{k=2}^{n} \mu(k)(I-S)h_k - (1-z) \Big\|_{X} \longrightarrow  0 \qquad \text{ as } n \to \infty.
	 \end{equation} 
 	We therefore call the invertibility of $I-S$ on $X$ the \textbf{Easy (C3)} condition. 
  
  We want to be clear that no new zero-free half-planes for $\zeta$ are obtained in this work. But rather, we hope it will lead to a deeper understanding of the challenges posed en route to such results, when viewed through the lens of spaces of analytic functions. It is reasonable to ask why $H^2(\mathbb{D})$ was chosen over $H^2(\mathbb{C}_\frac{1}{2})$ (see \cite{bagchi:paper}) or $\mathcal{A}$ (equivalently $\ell^2_w$ as in \cite{Balazard-Saias4}) to formulate our approach. Compared to the latter two spaces, the theory of $H^2(\mathbb{D})$ is the most complete of any reproducing kernel Hilbert space, with a vast array of tools and techniques developed over many decades. The article  \cite{noor:paper} contains some applications of these tools to the Nyman-Beurling approach to the RH.

\section{Weighted $ \ell^2 $ sequence spaces} 
\label{In Sequence Spaces}
In this section our goal  is to apply the fundamental principle of Section 2 to spaces $X$ of analytic functions with Taylor series coefficients in some weighted $ \ell^2 $ space. We begin with some illustrative examples.
\subsection{Some examples}\label{Some examples}
\begin{example}[Smaller disks] \label{example small disk}
	Recall the definition \ref{def H^2D} of $ H^2(\mathbb{D}) $.  We can restrict the supremum to $ 0 < r < \epsilon $ to get a Hardy Space on the smaller disk $ \mathbb{D}_\epsilon := B(0;\epsilon) $, where $ 0 < \epsilon < 1 $, defined by 
	\begin{equation} \label{eqn 4c}
		H^2(\mathbb{D}_\epsilon)  := \{f \in \mathrm{Hol}(\mathbb{D}_\epsilon) \mid \sup_{0<r<\epsilon} \frac{1}{2 \pi} \int_{0}^{2 \pi} |f(re^{i\theta})|^2 \dee \theta < \infty\}.
	\end{equation}
	It is a quick check to see this is equivalent to the weighted $ \ell^2 $ definition: 
	\begin{equation}\label{eqn 4d}
		H^2(\mathbb{D}_\epsilon) = \{\sum_{n = 0}^{\infty}a_n z^n \mid (a_n \epsilon^n)_{n \in \mathbb{N}} \in \ell^2\}.
	\end{equation}
	Comparing with the \hyperref[Checklist]{Checklist}, \textbf{(C1)} certainly holds as $ (z/\epsilon)^k $ form an orthonormal basis. From \eqref{eqn 4c}, we see that $ {H^2(\mathbb{D}) \subseteq  H^2(\mathbb{D}_\epsilon)} $  and \textbf{(C2)} holds. The \textbf{Easy (C3)} condition also follows by definition \ref{eqn 4c}: 
	\[\left\|\frac{1}{1-z}f(z)\right\|_{H^2(\mathbb{D}_\epsilon)} \leq \frac{1}{1-\epsilon}\left\|f(z)\right\|_{H^2(\mathbb{D}_\epsilon)}.\]
The problem arises in \textbf{(C4)}. That is because $ \Lambda^{(s)} $ is bounded on $H^2(\mathbb{D}_\epsilon)$ if and only if  $\Lambda^{(s)}(z^k/\epsilon^k)= f_k(s)/\epsilon^{k} $ forms an $\ell^2$ sequence. Since $ |f_k(s)|\asymp k^{-\Re(s)} $ by Proposition \ref{F_k estimate}, this requires that
	\[\left((1/\epsilon)^k k^{-\Re(s)}\right)_{k\in \mathbb{N^*}} \in \ell^2 \quad \text{where } (1/\epsilon)>1.\] 
	However, there are no values of $ s \in \mathbb{C} $ for which this is true as the  sequence above is unbounded. \qed
\end{example}
In the next example we see that \textbf{(C4)} does hold. 
\begin{example} \label{example logn}
	For $\alpha >0$, consider the space of analytic functions
	\begin{equation}\label{eqn 4e}
		X_\alpha := \{\sum_{n=0}^\infty a_n z^n \mid (a_n n^{-\alpha})_{n\in \mathbb{N^*}}  \in \ell^2 \}
	\end{equation}
	with the inner product $\langle\sum_{n=0}^\infty a_n z^n, \sum_{n=0}^\infty b_n z^n\rangle = a_0\overline{b_0} + \sum_{n=1}^{\infty} a_n \overline{b_n} n^{-2\alpha} $.	Condition \textbf{(C1)} holds as $ (n^\alpha z^n )_{n \in \mathbb{N^*} }\cup\{1\} $ forms an orthonormal basis for $X_\alpha$. Also, \textbf{(C2)} holds as $ H^2(\mathbb{D}) \subseteq X_\alpha$ and the norm on $X_\alpha$ is dominated by the $ H^2(\mathbb{D}) $ norm. Condition \textbf{(C4)} also holds: $\Lambda^{(s)} $ is bounded on $X_\alpha$ if and only if $\Lambda^{(s)}( n^\alpha z^n)=n^\alpha f_n(s) \in \ell^2$. Since $ |f_n(s)| \asymp n^{-\Re(s)}$, we require that
	\[\left(n^\alpha n^{-\Re(s)}\right)_{n\in \mathbb{N^*}} \in \ell^2 \quad \text{where } \alpha>0.\] 
	This holds for $ \Re(s)>\frac{1}{2} + \alpha $. It follows that if \textbf{(C3)} holds for any $0<\alpha<1/2$, then the checklist is satisfied by $X_\alpha$ and we obtain a non-trivial zero free half-plane for $\zeta$. But verifying \textbf{(C3)} for any such $\alpha$ is not easy. So lets consider \textbf{Easy (C3)} instead. Define the functions $ f_\delta(z) = \sum_{m=1}^{\infty} m^\delta z^m $ for some real $ \delta>0$. Then $ f_\delta \in X_\alpha $ if and only if $\delta<\alpha - \frac{1}{2}$. Since $(I-S)^{-1}$ is formally the operator of multiplication by $\frac{1}{1-z}$, we have
	\[(I-S)^{-1} f_\delta = \sum_{n=0}^{\infty}z^n \; \cdot \sum_{m=1}^{\infty} m^\delta z^m.\]
	By the ratio test we see that the infinite sum has radius of convergence $1$. Collecting the coefficient of $ z^k $ we see that $  (I-S)^{-1}f_\delta = \sum_{k=0}^{\infty}c_kz^k $ where $ c_k = \sum_{l=1}^{k} l^\delta $. Since $\delta>0$, we can bound this sum from below as follows
	\[ c_k = \sum_{l=1}^{k} l^\delta \geq \int_{0}^{k} x^\delta \dee x = \frac{k^{\delta+1}}{\delta +1} .\]
	Hence for $ (I-S)^{-1}f_\delta  \in X_\alpha $ we must necessarily have 
	\[\sum_{k=1}^{\infty}k^{2(\delta + 1)}k^{-2\alpha} < \infty\]
	\noindent or equivalently $ \delta <  \alpha - \frac{3}{2} $. Hence if we choose $\delta$ such that $\alpha-\frac{3}{2}<\delta<\alpha-\frac{1}{2}$, then  $f_\delta\in X_\alpha$ but $ (I-S)^{-1}f_\delta  \notin X_\alpha $. Therefore \textbf{Easy (C3)} fails for $X_\alpha$.\qed
\end{example}
\subsection{Analysis of weights}
We now consider the general setting of weighted $ \ell^2 $ sequence spaces. 

\noindent
	
		\begin{defn} For $ w_n \geq 1 \; (n\in\mathbb{N})$, we define the following Hilbert space: 
			\begin{equation}\label{def sequence space X}
				X = \{\sum_{n=0}^{\infty}a_n z^n \mid \left(a_n/w_n\right)_{n \in \mathbb{N}} \in \ell^2\}\simeq\ell^2_w
			\end{equation}
			where $ X $ is equipped with the inner product $$\langle\sum_{n=0}^\infty a_n z^n, \sum_{n=0}^\infty b_n z^n\rangle = \sum_{n=0}^{\infty} a_n \overline{b_n}/{w_n}^2.$$ Clearly $X$ is isometrically isomorphic to a weighted $ \ell^2 $-space.
		\end{defn}

It follows from the definition that $ (z^k w_k)_{k \in \mathbb{N}}$ forms an orthonormal basis for $ X $. As $ w_n \geq 1 $, the norm of $ X $ is dominated by that of $ H^2(\mathbb{D}) $ and $  H^2(\mathbb{D}) \subset X$. Therefore \textbf{(C1)} and \textbf{(C2)} both hold. We shall therefore focus on \textbf{Easy (C3)} and \textbf{(C4)}.

{ \noindent The examples we considered in Subsection \ref*{Some examples} are:
\begin{enumerate}
	\item $ w_n = 1  $ : Corresponds to $ X=H^2(\mathbb{D}) $. Here $ (I-S)^{-1} $ is not invertible, but $\Lambda^{(s)}$ is bounded for $ \Re(s) > 1/2 $. 
	\item $ w_n = (1/\epsilon)^n$ for $0<\epsilon<1$ : Corresponds to  $ H^2(\mathbb{D}_\epsilon)$ in Example~\ref{example small disk}. Here $I-S$ is invertible, but $\Lambda^{(s)}$ is not bounded for any $s$.
	\item $ w_n = n^\alpha  $, $\alpha > 0$ : Corresponds to $ X $ in Example~\ref{example logn}. Here $I-S$ is not invertible, but $\Lambda^{(s)}$ is bounded for $\Re(s)>1/2 + \alpha$.
\end{enumerate}
}

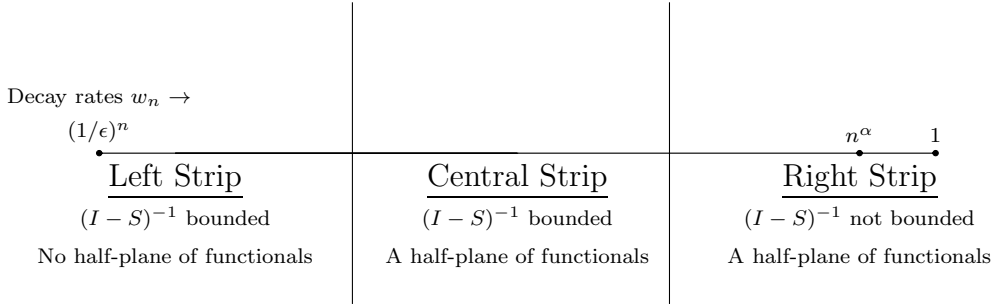
\begin{figure}[!h]
	\begin{tikzpicture}[node distance=7cm]
		\filldraw 
		
		(0,0) circle (1pt) node[align=center,   above] {\scriptsize $ (1/\epsilon)^n $} --
		(11,0) circle (1pt) node[align=center,   above] {\scriptsize $ 1 $}
		(10,0) circle (1pt) node[align=center,   above] {\scriptsize $ n^\alpha $} 
		(1,0) circle (0pt) node[align=center,   below] {\underline{Left Strip}\\ \scriptsize $(I-S)^{-1}$ bounded  \\ \scriptsize No half-plane of functionals} --
		(5.5,0) circle (0pt) node[align=center, below] {\underline{Central Strip} \\ \scriptsize $(I-S)^{-1}$ bounded  \\ \scriptsize A half-plane of functionals}      
		(10,0) circle (0pt) node[align=center,  below] {\underline{Right Strip}\\ \scriptsize$(I-S)^{-1}$ not bounded\\ \scriptsize A half-plane of functionals};
		\draw (3.33,-2)--(3.33,2);
		\draw (7.5,-2)--(7.5,2);
		\draw (0,1) node[align=center, below] { \scriptsize Decay rates $ w_n $ $\rightarrow$};
	\end{tikzpicture} 
	\caption{Weights and Conditions}
	\label{fig:Weights and Conditions AB}
\end{figure}

A general  tendency can be observed as depicted by Figure 2. If the decay rates of weights are too fast (left strip), we do not have a half-plane where $\Lambda^{(s)}$ is bounded. If the decay rates are too slow (right strip), then we do not have $ (I-S)^{-1} $ as a bounded operator. Ideally we would like to obtain weights that belong between these extremes (central strip).

 The next result characterizes the weights required for $X$ to satisfy \textbf{(C4)}.

		\begin{prop} \label{lem Condition B}
			Given a  weighted Hardy space $ X $ as in \eqref{def sequence space X} and $s\in\mathbb{C}$, the functional $\Lambda^{(s)}$ is bounded on $X$ if and only if 
			\begin{equation}\label{eqn Condition B}
				\left(\frac{w_k}{k^{\Re(s)}}\right)_{k\in \mathbb{N^*}} \in \ell^2 .
			\end{equation}
			\end{prop} 
\begin{proof}
	The monomials $(z^k w_k)_{k\in\mathbb{N}} $ form an orthonormal basis for $ X $. Hence $ \Lambda^{(s)} $ is bounded on $X$ precisely when $\Lambda^{(s)}(z^k w_k)=w_kf_k(s)$ forms an $\ell^2$ sequence. Therefore the estimate $|f_k(s)|\asymp k^{-\Re(s)}$ (Proposition \ref{F_k estimate}) implies that this is equivalent to 
	\[\left(\frac{w_k}{k^{\Re(s)}}\right)_{k\in \mathbb{N^*}} \in \ell^2. \] 
	This concludes the result.
\end{proof}

The following is a necessary condition for \textbf{Easy (C3)} to hold.

		\begin{prop} \label{lem Condition A}
			Given a  Hilbert Space $ X $ as in \eqref{def sequence space X}, let 
			\[ r_m  := \sum_{n=m}^{\infty} \frac{{w_m}^2}{{w_n}^2}.\]
			If $I-S$ is invertible on $X$, then $(r_m)_{m\in\mathbb{N}}$ is a bounded sequence.
		\end{prop}
	\begin{proof}
	\noindent Suppose the operator $(I-S)^{-1} $ is well defined and bounded, with operator norm $ C > 0 $. For each $ m \in \mathbb{N} $, consider $ (I-S)^{-1} z^m $ in the equation above. Then
		\[\|(I-S)^{-1} z^m\|_X^2 = r_m / {w_m}^2 = r_m \|z^m\|_X^2.\]
	By definition of the operator norm of $ (I-S)^{-1} $, $ r_m \leq C^2  $ for all $ m \in \mathbb{N} $
\end{proof}

Table 1 includes a collection of weights that have been tested.
\begin{table}[h!]
	
	\caption{Weights $ w_n $ and their classifications}
	\label{fig: table of classifications}
	\renewcommand{\arraystretch}{1}
	\centering
	\begin{center}
		\begin{tabular}[c]{|c |c|c|c|}
			\hline
			Weight $ w_n $& $I-S$ invertible & $\Lambda^{(s)}$ bounded & Strip \\
			\hline
			$ 1 $ & \xmark & \cmark & Right \\
			\hline
			$ n^\alpha $ & \xmark & \cmark & Right\\
			\hline
			$ n^\alpha + (\log n)^\beta$ & \xmark & \cmark &  Right \\
			\hline 
			$ \exp{((\log n)^{1+\alpha})}$, $ \alpha>0 $& \xmark & \xmark & - \\
			\hline
			$ \exp(n^\alpha) $, $ 0<\alpha<1 $ & \xmark & \xmark & -\\
			\hline
			$ (1/\epsilon)^n $ & \cmark & \xmark & Left \\
			\hline
			$  \exp(n^\alpha) $, $ \alpha>1 $ & \cmark & \xmark & Left \\
			\hline
		\end{tabular}
	\end{center}
	
\end{table}

\subsection{Extremal behavior of weights}

Suppose we have weights $ w_n $ for a space $X$ that satisfies both conditions \textbf{Easy (C3)} and \textbf{(C4)} (thus giving a zero-free half plane $ \Re(s)>r $). We are only interested in $\frac{1}{2}<r<1$ since $\zeta(s)$ has no  zeroes for $ r = \Re(s) \geq 1 $. The following result shows that such weights $w_n$ necessarily exhibit extremely divergent behavior.

\begin{prop} \label{prop limsup liminf}
	Let $ w_n $ be the weights for a sequence space $X$ satisfying conditions \textbf{Easy (C3)} and \textbf{(C4)} with $\frac{1}{2}<r<1$. If $ (n_i) \subseteq \mathbb{N} $ is a subsequence with $ \sum \frac{1}{n_i} = \infty $, then
	\begin{equation}\label{eqn limsupinf}
		\liminf_{i\to\infty}\frac{w_{n_i}}{n_i^{r-\frac{1}{2}}} = 0  \text{ \quad and \quad}
		\limsup_{i\to\infty}\frac{w_{n_i}}{n_i^{r-\frac{1}{2}}} = \infty.
	\end{equation}
	In particular, $ \lim_{i\to \infty} \frac{w_{n_i}}{n_i^{r-\frac{1}{2}}} $ does not exist for any such subsequence $ (n_i) $. 
\end{prop}
\begin{proof}
	We prove by contradiction for each half of the result. 
	
	Suppose, $ \liminf_{i\to\infty}\frac{w_{n_i}}{n_i^{r-\frac{1}{2}}} = C > 0 $. Then for any $ C^\prime < C$, there exists $ N \in \mathbb{N} $ such that for all  $ i\geq N $, $ \frac{w_{n_i}}{n_i^{r-\frac{1}{2}}} > C^\prime $. Thus, 
	\[\sum_{i=N}^{\infty} \left(\frac{w_{n_i}}{n_i^r}\right)^2 > C^{\prime 2}\sum_{i=N}^{\infty} \left(\frac{n_i^{r-\frac{1}{2}}}{n_i^r}\right)^2 = C^{\prime 2}\sum_{i=N}^{\infty}\frac{1}{n_i} = \infty. \]
	Hence we contradict Proposition \ref{lem Condition B} since $ (w_n/n^r)_{n>0} \in \ell^2$. 
	
	Now suppose, $ \limsup_{i\to\infty}\frac{w_{n_i}}{n_i^{r-\frac{1}{2}}} < +\infty $. Then, there is $ C^{\prime\prime} > 0 $ and $ M \in \mathbb{N} $ such that for all $ i>M $, $ \frac{w_{n_i}}{n_i^{r-\frac{1}{2}}} <  C^{\prime\prime} $. 
	This gives,
	\[\sum_{i=N}^{\infty} \frac{1}{w_{n_i}^2}> \frac{1}{C^{\prime\prime 2}} \sum_{i=N}^{\infty} \frac{1}{n_i^{2r-1}} > \frac{1}{C^{\prime\prime 2}} \sum_{i=N}^{\infty} \frac{1}{n_i} = \infty. \]
	Hence we contradict Proposition \ref{lem Condition A} since $ r_N $ is finite.
	\end{proof}

We note that Proposition \ref{prop limsup liminf} highlights a tension between \textbf{Easy (C3)} and \textbf{(C4)} which is supported by Figure 2 and Table 1.

\section{The Classical Hardy spaces $ H^p$ }
In this section we consider the spaces $H^p$ ($p>0$) consisting of functions $f$ holomophic in $\D$ for which
$$\norm{f}_p^p := \sup_{0<r<1} \int_\T \abs{f(rz)}^p \dd m(z)<\infty$$
where $m$ is normalized Lebesgue measure on $\mathbb{T}$. The text of Duren \cite{duren:hpspacesbook} is a classical reference. The $H^p$ spaces are Banach spaces for $p\geq1$ and complete metric spaces for $0<p<1$. By Fatou's theorem, any $f\in H^p$ has radial limits a.e. on $\mathbb{T}$ with respect to $m$. Using $f$  to also denote the radial limit function, we have
$$\norm{f}_p^p = \int_\T \abs{f(z)}^p \dd m(z).$$ 
The case $ p = 2 $ gives us the Hardy-Hilbert space $H^2(\mathbb{D}) $ as defined in \ref{def H^2D}. In this section we will focus on $ 0<p\leq 2 $. It is well-known that the monomials form a basis for $H^p$, that $H^p\subset H^q$ for $p>q$ and that the topology of $H^p$ weakens as $p$ decreases. Therefore the conditions \textbf{(C1)} and \textbf{(C2)} in the \emph{checklist} are satisfied for $X=H^p$ with $0<p\leq 2$. We shall see that \textbf{(C4)} also holds for all $0<p\leq 2$. As for \textbf{(C3)}, the next subsection uses the invertibility of $I-S$ between distinct $H^p$ spaces to prove \textbf{(C3)} when $0<p<1$. In Subsection 4.2 we show that \textbf{(C3)} for some $1<p\leq2$ would imply $\zeta(s)\neq 0$ for $\Re(s)>1/p$.  This is an $H^p$ analogue of Beurling's result \cite{beurling:closureproblempaper}.

\subsection{The $ H^p$ spaces for $ 0 < p < 1 $}
We first show that \textbf{(C4)} holds in this case.

\begin{prop}\label{boundedness 0<p<1}
$ \Lambda^{(s)}$ is bounded on $H^p$ for $ 0<p<1 $ if $ \Re(s) > \frac{1}{p} $.
\end{prop}
\begin{proof}
Let $ f(z) = \sum_{n=0}^{\infty} a_n z^n \in H^p $ for $ 0<p< 1 $. Then $ |a_n| \leq C n^{1/p - 1} \|f\|_{H^p}$ for some constant $ C >0 $ by \cite[Theorem 6.4]{duren:hpspacesbook}. Hence by Lemma \ref{F_k estimate}
	\[|\Lambda^{(s)}f| \leq \sum_{n=0}^{\infty} |a_n||f_n(s)| \leq C \sum_{n=0}^{\infty} n^{1/p - 1 - \Re(s)} \|f\|_{H^p}.\]
	So, $ \Lambda^{(s)} $ is bounded on $H^p$ if $ \Re(s) > \frac{1}{p} $.
\end{proof}

We now move to the proof of \textbf{(C3)}. We shall need  the following  result from Duren \cite[Theorem 6.1]{duren:hpspacesbook}.

\begin{thm}\label{Hp l^q norm relation}
	Let $1\leq q\leq2$ and $p$ satisfying $1/p+1/q=1$. If $(a_n)_{n\in \mathbb{N}}$ is a sequence in $\ell^q$, then $f(z) = \sum_{n=0}^\infty a_n z^n$ defines a function in $H^p$ satisfying
	$$\norm{f}_p \leq \norm{(a_n)_{n\in \mathbb{N}}}_q \,.$$
\end{thm}

We first extend the validity of equation \eqref{eq Noor convergence of (I-S)hk} to all $H^p$ spaces with $0<p<\infty$. Denote the $\ell^q$ norm  of $f(z) = \sum_n a_n z^n$ by $\norm{f}_{\ell_q} = (\sum_n \abs{a_n}^q)^{1/q}$. 

\begin{lem}\label{Noor equation all Hp}
For all $0<p<\infty$, we have
\begin{equation} \label{Eq:sum_gk_Hp}
\sum_{k=2}^n \mu(k)(I-S) h_k \to 1-z \ \text{ in } H^p.
\end{equation}
\end{lem}

\begin{proof} First note that we already  have the result for $0<p\leq2$ by \eqref{eq Noor convergence of (I-S)hk} which corresponds to $2\leq q<\infty$. Therefore by Theorem  \ref{Hp l^q norm relation} it is enough to prove that $\sum_{k=2}^n \mu(k) (I-S) h_k \to 1-z$ in the $\ell^q$ sense for $1 < q <2$. We have
	\begin{equation}\label{g_k  Taylor} \sum_{k=2}^n \mu(k)(I-S) h_k(z) = \sum_{k=1}^n \frac{\mu(k)}{k} \left[ \log(1-z^k) - \log(1-z) - \log k \right]. 
	\end{equation} 
	The Taylor coefficients of $\log(1-z) = -\sum_{j=1}^\infty z^j/j$ belong to $\ell^q$ for $1 < q <2$, which implies the same for $\log(1-z^k) = -\sum_{j=1}^\infty z^{jk}/j$. Therefore the Taylor coefficients of\eqref{g_k  Taylor} belong to $\ell^q$ for each $n\geq 2$. We shall need the following relations involving the Möbius function
	\begin{equation}\label{PNT mu}
	\sum_{k=1}^\infty \frac{\mu(k)}{k}=0 \ \ \mathrm{and} \ \ \sum_{k=1}^\infty\frac{\mu(k)\log k}{k}=-1
	\end{equation}
	(see \cite[Thm. 4.16]{ApostolTomM1976Itan} and \cite[p. 185, Excercise 16]{MontgomeryHughL.2007MntI}). This implies that $$\sum_{k=1}^n \frac{\mu(k)}{k}[-\log(1-z)-\log k] \to 1$$ in $\ell^q$ norm as $n\to\infty$. Hence by \eqref{g_k  Taylor} it suffices to prove that
	\begin{equation} \label{Eq:sum_log_lq}
	\sum_{k=1}^n \frac{\mu(k)}{k} \log(1-z^k) \to -z \ \text{ in the } \ell^q \text{ sense.}
	\end{equation}
	By \cite[eq. (4.6)]{noor:paper} we have
	$$\sum_{k=1}^n \frac{\mu(k)}{k} \log(1-z^k) + z = -\sum_{j=n+1}^\infty \frac{z^j}{j} \sum_{\substack{d|j \\ 1\leq d\leq n}} \mu(d) \,$$
	with
	$$\abs{\sum_{\substack{d|j \\ 1\leq d\leq n}} \mu(d)} \leq \sum_{\substack{d|j \\ 1\leq d\leq n}} \abs{\mu(d)} \leq \sum_{d|j} 1 = \tau(j) \,,$$
where $\tau(n)$ denotes the number of divisors of $n$. Since $\tau(n)=o(n^{\epsilon})$ for every $\epsilon>0$ \cite[p. 296]{ApostolTomM1976Itan}, we get
	$$\lnorm{\sum_{k=1}^n \frac{\mu(k)}{k} \log(1-z^k) + z}_{\ell_q}^q
	= \sum_{j=n+1}^\infty \frac{1}{j^q} \abs{\sum_{\substack{d|j \\ 1\leq d\leq n}} \mu(d)}^q
	\leq \sum_{j=n+1}^\infty \frac{\sigma(j)^q}{j^q}\rightarrow 0
	$$
	since $\sigma(j)\lesssim j^\epsilon$ where $\epsilon>0$ can be chosen  small enough so that $q-\epsilon q>1$. This proves \eqref{Eq:sum_log_lq} and hence the result.
\end{proof}
Although the $I-S$ is not invertible on any $H^p$ space, $(I-S)^{-1}=M_{\frac{1}{1-z}}$ may still be bounded between different $H^p$ spaces.
\begin{lem}\label{Inverting I-S between Hp spaces} For each $0<q<1$ there exists $p>0$ sufficiently large such that the  operator $M_{\frac{1}{1-z}}$ from $H^p$ to $H^q$ is bounded.	
\end{lem}
\begin{proof}
	We shall need the reverse H\"older's inequality: Let $0<r<1$ and $s$ satisfying $\frac{1}{r}+\frac{1}{s}=1$ (so that $s<0$). For any non-negative $f \in L^r(\mathbb{T})$ and $g \in L^s(\mathbb{T})$ with $\int g^s\dd m > 0$,
	$$\int fg \dd m \geq \left( \int f^r \dd m\right)^{1/r} \left( \int g^s \dd m\right)^{1/s} \,.$$
	Choose any $0<q<p$ (not necessarily conjugate exponents) and let $h \in H^p$. Define $r:=q/p<1$ (hence $s<0$), $f(z) := \abs{h(z)/(z-1)}^p$ and $g(z) := \abs{1-z}^p$. So we have
\[
	\int f^r\dd m = \int_\T \labs{\frac{h(z)}{z-1}}^q \dd m, \ \ 
		\int g^s\dd m = \int_\T \abs{1-z}^{ps} \dd m,\ \ \int fg\dd m = ||h||_p^p.
		\]
		Therefore we get
	\begin{equation}\label{1/1-z bounded}
	\left(\int_\T \labs{\frac{h(z)}{z-1}}^q \ddd m\right)^{1/r}\leq ||h||_p^{p}\left(\int_\T \abs{1-z}^{ps} \dd m\right)^{-1/s} .
	\end{equation}
	For the right side of \eqref{1/1-z bounded} to be finite,  we need $\int_\T \abs{1-z}^{ps} \dd m<\infty$ keeping in mind that $s<0$. This occurs precisely when
	$$ps>-1 \iff -\frac{1}{s}>p \iff \frac{1}{r}>1+p \iff \frac{q}{p}<\frac{1}{1+p} \iff q<\frac{p}{1+p} \,.$$
	So we  necessarily have $0<q<1$. Hence we conclude that for any $0<q<1$ there exists $p>0$ large enough satisfying $q<\frac{p}{1+p}$ for which \eqref{1/1-z bounded} gives 
	\begin{equation} \label{Eq:1-z_Hp_Hq}
		\lnorm{\frac{h}{1-z}}_q\leq  C_{p,q}\norm{h}_p  \ \ \ \forall \ \ h\in H^p
	\end{equation}
	where constant $C_{p,q}>0$ depends only on $p$ and $q$.  This proves the lemma.
	\end{proof}
We are now ready to prove \textbf{(C3)} for $X=H^q$ with $0<q<1$.
\begin{thm}\label{Approximation  in H^p p<1}
$ \sum_{k=2}^n \mu(k) h_k \to 1 $ in $H^q$ for $0<q<1$.
\end{thm}
\begin{proof} For any $0<q<1$ there exists a $p>0$ large enough so that $M_{\frac{1}{1-z}}$ is bounded from $H^p$ to $H^q$ by Lemma \ref{Inverting I-S between Hp spaces}. Therefore applying $M_{\frac{1}{1-z}}$ to the approximation  $\sum_{k=2}^n \mu(k) (I-S) h_k \to 1-z$ in $H^p$ from Lemma \ref{Noor equation all Hp} gives the result.
\end{proof}

Therefore the checklist is completely satisfied for $H^p$ with $0<p<1$ giving the zero-free half-plane $ \Re(s) > 1 $. As an immediate corollary we get


\begin{cor}  \label{prop density}
$\mathrm{span}(h_k)_{k\geq 2}$ is dense in $ H^p $ for all $0<p<1$.
\end{cor}
\begin{proof}
	We need to show that the weighted composition operators
	$$W_nf(z) = (1+z+\cdots+z^{n-1}) f(z^n),\qquad n\geq1$$
	introduced  in \cite{noor:paper} are bounded in $H^p$. The Littlewood Subordination Theorem \cite[Theorem 1.7]{duren:hpspacesbook} states that if $\varphi \in \hol(\D)$, then
	$$\ \abs{\varphi(z)}\leq\abs{z}\ \forall z\in\D \implies \int_\T \abs{f\circ\varphi}^p \dd m \leq \int_\T \abs{f}^p \dd m$$
	for all $p\in(0,\infty]$. So  the operator $f \mapsto f \circ \varphi$ with $\varphi(z) = z^n$ is bounded on $H^p$ as is the multiplication operator $f \mapsto \psi f$ where $\psi(z) = 1+z+\cdots+z^{n-1}$. Therefore similar to \cite[Section 3]{noor:paper}, the bounded semigroup $(W_n)_{n\in\mathbb{N}}$ leaves $\overline{\mathrm{span}}_{H^p}\{h_k\mid k \geq 2\}$ invariant and it contains the constant $1$ by Theorem \ref{Approximation  in H^p p<1} for $0<p<1$. But $1$ is a cyclic vector for $(W_n)_{n\in\mathbb{N}}$ since $\mathrm{span}(W_n1)_{n\in\mathbb{N}}$ contains all analytic polynomials and is hence dense in $H^p$ for $0<p<1$.
\end{proof}
\subsection{Zero free half-planes via $ H^p $ spaces }
Our main goal here is to show that condition \textbf{(C4)} holds for $H^p$ with $ 1\leq p\leq2 $ and therefore that proving \textbf{(C3)} immediately provides nontrivial zero free half-planes for $\zeta$. Recall that for each $s\in\mathbb{C}$, the linear functionals $\Lambda^{(s)}:X\to\mathbb{C}$ are formally defined by 
\[
\Lambda^{(s)}(z^n)=f_n(s)=-\frac{1}{s}((n+1)^{1-s}-n^{1-s})
\]
where $|f_n(s)|\asymp n^{-\Re(s)}$ for $n\in\mathbb{N}$ by Proposition \ref{F_k estimate}. \\ 
\begin{prop}\label{Boundedness 1<p<2} If $1\leq p\leq 2$, then $\Lambda^{(s)}$ is bounded on $H^p$ for $\Re(s)>\frac{1}{p}$ and is bounded on $H^1$ for $\Re(s)\geq 1$.
\end{prop}
To prove this we need the following results from Duren's book \cite{duren:hpspacesbook}: \\ \\
	$(\mathbf{a})$ \textit{(\cite[Theorem 3.15]{duren:hpspacesbook}}) If $f(z)=\sum a_n z^n\in H^1$, then
	\[
	\sum_{n=0}^\infty\frac{|a_n|}{n+1} \leq \pi||f||_1.
	\]
	$(\mathbf{b})$\textit{(\cite[Theorem 6.3]{duren:hpspacesbook}})  If $(a_n)$ is a sequence such that \[\sum_{n=0}^\infty n^{q-2}|a_n|^q<\infty\]
	for some $2\leq q<\infty$, then $f(z)=\sum a_n z^n\in H^q$.\\ \\
	$(\mathbf{c})$ \textit{(\cite[Theorem 7.3]{duren:hpspacesbook}})  For $1<p<\infty$, each $\phi\in(H^p)^*$  is representable in the form
	\[
	\phi(f)=\frac{1}{2\pi}\int_0^{2\pi}f(e^{i\theta})\overline{g(e^{i\theta}})d\theta \ \ \ \mathrm{for} \ \ f\in H^p
	\]
	by a unique $g\in H^q$ where $\frac{1}{p}+\frac{1}{q}=1$.

	\begin{proof} Let $f(z)=\sum a_n z^n\in H^p$ and let $p\in (1,2]$. Define the functions
	\[
	k_s(z)=\sum_{n=0}^\infty f_n(s)z^n \ \ \mathrm{for} \  \mathrm{each} \  s\in\mathbb{C}.
	\]
	Then $q=\frac{p}{p-1}\geq 2$ and we get
	\[
	\sum_{n=1}^\infty n^{q-2}|f_n(s)|^q\leq C\sum_{n=1}^\infty n^{q-2-q\Re(s)}<\infty
	\]
	if $q-2-q\Re(s)<-1$ or equivalently if $\Re(s)>\frac{q-1}{q}=\frac{1}{p}$. So $(\mathbf{b})$ implies that $k_s\in H^q$ for $\Re(s)>\frac{1}{p}$. Therefore the functional $\phi_s$ defined by
	\[
	\phi_s(f)=\frac{1}{2\pi}\int_0^{2\pi}f(e^{i\theta})\overline{k_{\bar{s}}(e^{i\theta}})d\theta
	\]
	is bounded on $H^p$ for $\Re(s)>\frac{1}{p}$ by $(\mathbf{c})$. Now since $k_s\in H^q\subset H^2$ for $q\geq 2$, we see that
	\[
	\phi_s(z^n)=\left<z^n,k_{\bar{s}}\right>=\overline{f_n(\bar{s})}=f_n(s)=\Lambda^{(s)}(z^n) 
	\]
	for all $n\in\mathbb{N}$ and hence $\Lambda^{(s)}=\phi_s$ and $\Lambda^{(s)}\in(H^p)^*$ for $\Re(s)>\frac{1}{p}$ when $p\in (1,2]$. For $p=1$ and $\Re(s)\geq 1$, $(\mathbf{a})$ gives
	\[
	|\Lambda^{(s)}f|\leq\sum_{n=0}^\infty|a_n||f_n(s)|\leq\frac{|a_0|}{|s|}+\sum_{n=1}^\infty\frac{|a_n|}{n^{\Re(s)}}\leq \sum_{n=0}^\infty\frac{2|a_n|}{n+1}\leq 2\pi||f||_1
	\]
	and hence $\Lambda^{(s)}$ is bounded on $H^1$ for $\Re(s)\geq 1$.
\end{proof}
Therefore proving condition \textbf{(C3)} for $H^p$ with $1<p\leq 2$ will lead to nontrivial zero free half-planes for $\zeta$.
\begin{thm}\label{Half-plane 1<p<2} For any $1<p\leq 2$, we have
	\[
	1 \in \overline{\mathrm{span}}_{H^p}\{h_k\mid k \geq 2\} \implies \zeta(s) \neq 0  \ \  \mathrm{for} \ \   \Re(s)> 1/p .
	\]
	Note that the case $p=1$ gives the known zero-free half-plane $\Re(s)\geq 1$ and the hypothesis above is equivalent to the density of $\mathrm{span}(h_k)_{k\geq 2}$ in $ H^p$ by Theorem \ref{prop density}. It is unclear whether the converse of Theorem \ref{Half-plane 1<p<2} holds.
	\end{thm}

\section*{Acknowledgment}
We are grateful to the anonymous referee for providing detailed comments and suggestions which have improved this work considerably. This work was financed in part by the Coordenação de Aperfeiçoamento de Nivel Superior- Brasil (CAPES)- Finance Code 001.

\bibliographystyle{siam}
\bibliography{References}

\end{document}